\documentclass[12pt]{article}
\usepackage{amsfonts}
\usepackage{amscd}
\usepackage{amsmath}
\usepackage{amsthm}
\usepackage{amssymb}
\usepackage{amsxtra}
\input xy
\xyoption{all}

\normalsize

\newcommand{\bA}{{\mathbb A}}

\newcommand{\caD}{{\mathcal D}}

\newcommand{\caS}{{\mathcal S}}

\newcommand{\caT}{{\mathcal T}}

\newcommand{\integers}{{\mathbf Z}}
\newcommand{\complex}{{\mathbf C}}
\newcommand{\affine}{{\mathbf A}}

\newcommand{\MM}{{\mathsf M}}
\newcommand{\NN}{{\mathsf N}}
\newcommand{\CC}{{\mathsf C}}
\newcommand{\HH}{{\mathsf H}}
\newcommand{\EE}{{\mathsf E}}
\newcommand{\FF}{{\mathsf F}}
\renewcommand{\S}{{\mathsf S}}

\newcommand{\SH}{\mathsf{SH}}
\newcommand{\MU}{\mathsf{MU}}

\newcommand{\MGL}{\mathsf{MGL}}
\newcommand{\KGL}{\mathsf{KGL}}

\newcommand{\Spec}{\mathrm{Spec}}
\newcommand{\Hom}{\mathrm{Hom}}

\newcommand{\BAb}{\mathsf{BAb}}

\newcommand{\one}{\mathbf{1}}
\newcommand{\Top}{\mathrm{Top}}
\newcommand{\rig}{\mathrm{rig}}
\newcommand{\eff}{\mathrm{eff}}

\newcommand{\SlMGL}{\mathrm{(SlMGL)}}

\newtheorem{lem}{Lemma}[section]
\newtheorem{thm}[lem]{Theorem}
\newtheorem{cor}[lem]{Corollary}
\newtheorem{prop}[lem]{Proposition}

\newtheorem{rem}[lem]{Remark}

\begin{document}
\title{\bf Slices of motivic Landweber spectra}
\author{Markus Spitzweck}
\date{April 2009}

\maketitle

\tableofcontents

\section{Introduction}

In this paper we show that a Conjecture of Voevodsky
concerning the slices of the motivic Thom spectrum
$\MGL$ \cite{voe-slice} implies a general statement about slices
of motivic Landweber spectra.

A similar result is announced by Hopkins-Morel.

A proof of Voevodsky's conjecture, to the
author's knowledge over fields
of characterisic $0$, is likewise announced by Hopkins-Morel.

In \cite{levine-htp} Levine gives an unconditional computation
of the slices of the algebraic $K$-theory spectrum
$\KGL$ over perfect fields
yielding (shifted) motivic Eilenberg-MacLane spectra.

In \cite{voe-slice} it is suggested that a Conner-Floyd type
isomorphism $$\KGL_{**}(X) \cong \MGL_{**}(X) \otimes_{\MU_*}
\integers[u,u^{-1}]$$
for homotopy algebraic $K$-theory
could yield a proof of the conjectures on the slices of
$\KGL$ in \cite{voe-slice} assuming the conjectures about $\MGL$.

Using the Conner-Floyd isomorphism for homotopy algebraic $K$-theory
established in \cite{oestvaer-spitzweck} and in \cite{PPR3}
over fields in a slightly different form
our result gives a positive answer to the strategy suggested
in \cite{voe-slice}.

We point out that Voevodsky's full Conjecture
uses the motivic Eilenberg-MacLane spectrum,
in particular the conjecture says that the zeroth slice
of $\MGL$ is the motivic Eilenberg-MacLane spectrum.

In our discussion we formulate the conjecture for all slices of
$\MGL$ relative to the zeroth slice.

The zeroth slice of the sphere spectrum is known
to be motivic cohomology over perfect fields (see \cite{voevodsky-zero-slice}
for fields of characteristic $0$ and \cite{levine-htp} for perfect fields).
By \cite[Corollary 1.3]{spitzweck-rel} the zeroth' slices of
the sphere spectrum and $\MGL$ agree.

The proof of the main result consists of two steps.
In the first we show that the statement holds for Landweber exact
spectra of the form $\MGL \wedge \EE$, $\EE$ also a Landweber
spectrum. The main ingredient is a topological result about
the projective dimension of the $\MU$-homology of an even
topological Landweber spectrum, \cite{hov-strick.mor}.

In the second step we use a cosimplicial resolution of the
given Landweber spectrum in terms of spectra of the form
from the first step.

In a last paragraph we show that the argument used here also
shows that cohomological Landweber exactness holds for all compact spectra,
not only for the strongly dualizable ones as shown in \cite{NOS-land}.

\vskip.4cm

{\bf Acknowledgements:} I would like to thank Niko Naumann, Paul Arne
{\O}stv{\ae}r,
Marc Levine and Fabien Morel for helpful
discussions on the subject.

\section{Conventions}

We fix a base scheme $S$ (Noetherian of finite Krull dimension)
and denote the stable $\bA^1$-homotopy category over $S$ by
$\SH(S)$. As in \cite{NOS-land} we let $\SH(S)_\caT$ be the full localizing
triangulated category spanned by all Tate spheres $S^{p,q}$ which is also known
as cellular spectra.

We let $\one$ be the motivic sphere spectrum.

Generalizing the notion of rigid homotopy groups
of a spectrum $\EE$ given in \cite{voe-slice}
(i.e. $\pi^\rig_{p,q}(\EE)=\pi_{p,q} s_q \EE$)
we set $\pi^\rig_{p,q}(X_\bullet)= \pi_{p,q}(X_q)$
for an object
$X_\bullet \in \SH(S)^\integers$.

We set $\pi^\rig_p(X_\bullet) := \pi^\rig_{2p,p}(X_\bullet)$.

\section{Preliminaries on the slice filtration}

As in \cite{voe-slice} we denote the
slices of a spectrum $X \in \SH(S)$ by $s_i(X) \in \SH(S)$.

The functor
$$s_* \colon \SH(S) \to \SH(S)^\integers$$
has good multiplicative properties, for a general
treatment of that using the theory of model categories see \cite{pelaez}.
It is in particular shown in loc. cit. that the functor $s_*$
preserves ring and module objects in a highly structured sense.
For most of the paper we use these statements
on the level of homotopy, see \cite[page iv, (5)]{pelaez}.

In \cite[p. 5]{voe-slice} it is observed that there are
natural maps $s_i(\EE) \wedge s_j(\FF) \to s_{i+j}(\EE \wedge \FF)$
(the map in loc. cit. is written after taking sums over all $i$ resp. $j$).
Assembling these maps in a graded way gives natural maps
in $\SH(S)^\integers$
$$\alpha_{\EE,\FF} \colon s_*(\EE) \wedge s_*(\FF) \to s_*(\EE \wedge \FF),$$
where the $\wedge$-product in $\SH(S)^\integers$ is defined using
the $\wedge$-product in $\SH(S)$ and taking sums of $\wedge$-products
of a fixed total degree.

Indeed the $\alpha_{\EE,\FF}$ assemble to give $s_*$ the structure
of a lax tensor functor by the following argument:

The slice  $s_0(\EE)$ for an effective spectrum $\EE \in \SH(S)^\eff$
can be obtained by a left Bousfield localization
of the triangulated category $\SH(S)^\eff$ along the subcategory
$\Sigma_T \SH(S)^\eff$, see \cite{pelaez} for a model category version of this.
In detail the functor $s_0$ restricted to effective objects is the composition
$$\SH(S)^\eff \to \SH(S)^\eff/\Sigma_T \SH(S)^\eff \to \SH(S)^\eff,$$
where the first arrow is the quotient map and the second arrow
the right adjoint to the quotient map which exhibits the quotient
as a full subcategory of the first category. For the existence of the
quotients see e.g. \cite[par. 5.6]{krause}.
Now $\Sigma_T \SH(S)^\eff$ is a tensor ideal of $\SH(S)^\eff$,
therefore there is an induced $\wedge$-product on the quotient
$\SH(S)^\eff/\Sigma_T \SH(S)^\eff$ and the quotient map
$\SH(S)^\eff \to \SH(S)^\eff/\Sigma_T \SH(S)^\eff$ is a tensor functor.
Thus the right adjoint is a lax tensor functor, which gives
$s_0 \colon \SH(S)^\eff \to \SH(S)^\eff$ the structure
of a lax tensor functor. By applying suitable shifts
$\Sigma_T^i$ this construction gives the functor
$s_* \colon \SH(S) \to \SH(S)^\integers$ the structure of
a lax tensor functor.

In the whole paper we will denote the spectrum
$s_0(\MGL)$ by $\HH$. By the above it is a ring spectrum and
using the effectivity of $\MGL$ (\cite[Corollary 3.2]{spitzweck-rel})
it comes with a ring map $\MGL \to \HH$.

We make the following assumption,
which is called $\SlMGL$:

$s_i(\MGL) \cong \Sigma_T^i \HH \otimes \pi_{2i}(\MU)$
in $\SH(S)$
compatible with the homomorphism $\MU_* \to \MGL_{**}$
as in \cite[Conjecture 5]{voe-slice}.

The assumption implies that
shifted slices $\Sigma^{0,-i} s_i \MM$
of a cellular $\MGL$-module $\MM$ are in the localizing
triangulated subcategory of $\SH(S)$
generated by $\HH$. We call such spectra {\em strictly $\HH$-cellular}.
We call a module $X_\bullet \in \SH(S)^\integers$
stricly $\HH$-cellular if for all $i$ the module
$\Sigma^{0,-i} X_i$ is strictly $\HH$-cellular.

\section{Remarks on phantom maps}

Throughout the paper we will use the notion of {\em phantom map}.
We recall that in a compactly generated triangulated category with sums
a map between two objects is called
{\em phantom} if it induces the zero map between the represented
cohomology theories on compact objects.

If the triangulated category has a compatible tensor product
and if every compact object is strongly dualizable then this is the same
that the corresponding homology transformation on the whole category,
or equivalently on the compact objects, is zero.

This is the case e.g. for the categories $\SH(S)_\caT$, $\SH(S)_\caD$
(the last category is spanned by strongly dualizable objects, see \cite[par. 4]{NOS-land}).

Let $f \colon T \to S$ be a map between base schemes. 
Let $g$ be the right adjoint to the pullback functor $f^* \colon \SH(S)_\caT \to \SH(T)_\caT$.
Then $g$ is a $\SH(S)_\caT$-module
functor (compare \cite[Lemma 4.7]{NOS-land}). Let $F \colon \EE \to \FF$ be a phantom map
in $\SH(S)_\caT$. Then $g$ applied to $f^*\EE \wedge K \to f^* \FF \wedge K$
yields $\EE \wedge g(K) \to \FF \wedge g(K)$. It follows that $f^* \colon \SH(S)_\caT \to \SH(T)_\caT$
preserves phantom maps. A similar argument shows that $\SH(S)_\caT \hookrightarrow \SH(S)_\caD$
preserves phantoms.

Let $F$ be as above. One can also see that for a smooth $S$-scheme $X$ the transformation
$\Hom([X],F)$ is zero. It is not clear to the author if $F$ is necessarily phantom in $\SH(S)$.

\section{Landweber spectra}

We recall briefly some results from
\cite{NOS-land} which we shall need in this paper.

The motivic Thom spectrum $\MGL$
is a commutative monoid in $\SH(S)$.
By the construction of \cite[2.1]{PPR2} there is a strictly
commutative model as symmetric $T$-spectrum,
$T$ the Tate object $\affine^1/(\affine^1 \setminus \{0\})$.

We let $\BAb$ be the abelian category of bigraded abelian
groups.

For a Landweber exact $\MU_*$-module $M_*$ (which we always consider
to be evenly graded in the usual topological grading, but we adopt
the convention that we regrade by dividing by $2$)
one looks at the functor
$$\begin{array}{lll}
\SH(S) & \to & \BAb \\
X & \mapsto & \MGL_{**}(X) \otimes_{\MU_*} M_*.
\end{array}$$
Here $\MU_*$ and $M_*$ are considered as bigraded
(more precisely Adams graded graded) abelian groups
via the diagonal $\integers(2,1)$
(for more precise statements see \cite{NOS-land}).
By the results of \cite{NOS-land} this functor
is a homology theory on $\SH(S)$ and
representable by a cellular (or Tate-) spectrum $\EE$.
There is a choice of that spectrum which is canonical up to isomorphism
(which is canonical up to a possible phantom map in Tate-spectra)
by requiring that $\EE$ is the pullback of a Tate-spectrum representing
the same theory over the integers.

A refined version of this statement gives a representing
object as highly structured $\MGL$-module.

Let $\caD_\MGL$ be the derived category of (highly structured)
$\MGL$-modules. Then the functor
$$\begin{array}{lll}
\caD_\MGL & \to & \BAb \\
X & \mapsto & X_{**} \otimes_{\MU_*} M_*.
\end{array}$$
is a homology theory and representable by a cellular
$\MGL$-module.

We let $\caD_{\MGL,\caT}$ be the subcategory of cellular $\MGL$-modules.

\section{Slices of Landweber exact theories}

\begin{thm} \label{sl-thm}
Suppose $\SlMGL$ is fulfilled.
Let $M_*$ be a Landweber exact $\MU_*$-module
and let $\EE_\integers$ be the corresponding
Landweber exact motivic spectrum in $\SH(\integers)$
given by \cite[Theorem 9.7]{NOS-land}. Let
$\EE$ be its pullback to $S$. Then
$s_i(\EE) \cong \Sigma_T^i \HH \otimes M_i$ (here $M_i$ is the $2i$-th
homotopy group of the corresponding topological
Landweber spectrum) compatible with the
homomorphism $M_* \to \EE_{**}$.
\end{thm}

In the above $\HH \otimes A$ for a torsion free abelian group
$A$ is the spectrum $\HH \otimes (\S^\Top \otimes A)$,
where the first $\otimes$ is the exterior action
of the stable topological homotopy category and
$\S^\Top \otimes A$ is the sphere spectrum with
$A$-coefficients, i.e. a spectrum representing
the homology theory $X \mapsto X_0 \otimes A$
on the topological stable homotopy category.
$\S^\Top \otimes A$ is well defined up
to possible phantom maps.

\begin{cor}
Suppose $\SlMGL$ is fulfilled. Then $s_i(\KGL) \cong
\Sigma_T^i \HH$ compatible with the natural map $\integers \to \pi_{2i,i}
\KGL$.
\end{cor}
\begin{proof}
The spectrum $\KGL$ is Landweber exact for the $\MGL$-algebra
$\integers[u,u^{-1}]$ classifying the mutliplicative formal
group law over $\integers[u,u^{-1}]$, see \cite[Theorem
1.2]{oestvaer-spitzweck}. The result follows from Theorem
\ref{sl-thm}.
\end{proof}

\begin{lem}
\label{tors-ten}
Let $R$ be a motivic ring spectrum (i.e. a commutative
monoid in $\SH(S)$), $A$ a torsion free abelian group, $M$ a
$R$-module and $\varphi \colon A \to \pi_{0,0} M$ a map.
Then there is a map $R \otimes A \to M$ which is an $R$-module
map and which induces $\varphi$
via $A \to \pi_{0,0} (R \otimes A) \to \pi_{0,0} M$. Moreover
it is well defined up to phantoms in $\SH(S)$.
\end{lem}
\begin{proof}
First note that $\one \otimes A$ has such a universal property
by using the adjunction $\SH \to \SH(S)$, $\SH$ the topological
stable homotopy category and the corresponding universal property
of $\S^\Top \otimes A$. Tensoring the resulting map
$\one \otimes A \to M$ with $R$ and composing with the module structure
map gives the required map. It is unique up to phantoms since
on the level of cohomology theories on compacts it is well defined.
\end{proof}

\subsection{Slices of Landweber spectra of the form $\MGL \wedge \EE$} \label{first-sl}

One idea is to use resolutions of the $\MU_*$module $M_*$
by free or projective $\MU_*$-modules. Let $M_*$ be the coefficients
of a Landweber spectrum $\MU \wedge \EE^\Top$ for $\EE^\Top$ also
Landweber.
Here we induce the $\MU_*$-module structure
from the first factor in $\MU \wedge \EE^\Top$.
We let $\EE_\integers$ be the $\MGL_\integers$-module
representing the theory for the module $\EE^\Top_*$.
Hence $\MGL_\integers \wedge \EE_\integers$
represents the theory corresponding to $M_*$.

By \cite[2.12 and 2.16]{hov-strick.mor} there exists a $2$-term resolution
of $M_*$ by projective $\MU_*$-modules
\begin{equation} \label{M-reso}
0 \to P_* \overset{\phi}{\to} Q_* \to M_* \to 0,
\end{equation}
where $P_*$ and $Q_*$ come by construction
as retracts of free $\MU_*$-modules
(see \cite[Lemma 4.6]{christensen-strickland} which is
cited in the proof of  \cite[2.14]{hov-strick.mor}), say of
$\bigoplus_i \MU_*(n_i)$ and $\bigoplus_j \MU_*(m_j)$.

As $\MU_*$-module $M_*$ is flat. We shall not need this fact in this paragraph,
it will become relevant in the last paragraph.

For any Landweber exact $\MU_*$-module $N_*$
(in particular for any projective $\MU_*$-module) we denote
by $h_{M_*}$ the corresponding homology theory
on $\caD_{\MGL_\integers}$ given by $X \mapsto (X_* \otimes_{\MU_*} N_*)_0$.
Any $\MU_*$-module map between such modules gives rise
to a transformation betweeen the homology theories.

Hence we get the sequence
\begin{equation} \label{short-exact-hom}
0 \to h_{P_*} \to h_{Q_*} \to h_{M_*} \to 0
\end{equation}
of homology theories.

This is short exact
since by the flatness of $M_*$ as quasi coherent sheaf over
the moduli stack of formal groups with trivialized constant vector fields for any
$X \in \caD_{\MGL_\integers}$ the map
$h_{P_*}(X) \to h_{Q_*}(X)$ is a mono (*).
Now lift $h_\phi \colon h_{P_*} \to h_{Q_*}$
to a map between cellular $\MGL_\integers$-modules $\Phi \colon \MM_P \to \MM_Q$.
($P_*$ and $Q_*$ are projective so this is easy,
one can also invoke that $\caD_{\MGL_\integers,\caT}$ is Brown.)

Let $\CC_\integers$ be the cofiber of $\Phi$. By (*)
the sequence of homology theories associated to
the exact triangle
\begin{equation} \label{aufl-tri}
\MM_P \to \MM_Q \to \CC_\integers \to \MM_P[1]
\end{equation}

is isomorphic to the sequence (\ref{short-exact-hom}),
in particular the homology theory associated to $\CC_\integers$ is canonically
isomorphic to $h_{M_*}$. Hence $\CC_\integers$
is isomorphic to $\MGL_\integers \wedge \EE_\integers$
since $\caD_{\MGL_\integers,_\caT}$ is Brown.

We now look at the triangle
\begin{equation}
s_*(\MM_{P,S}) \to s_*(\MM_{Q,S}) \to s_*(\CC_S) \to s_*(\MM_{P,S})[1]
\end{equation}
in $\SH(S)^\integers$, $\MM_{P,S}$, $\MM_{Q,S}$, $\CC_S$ the pullbacks
of
$\MM_P$, $\MM_Q$, $\CC_\integers$ to $S$.

Since we have maps $P_* \to \MM_{P,S,*}$, $Q_* \to \MM_{Q,S,*}$,
$M_* \to \CC_{S,*}$ we get maps
$$P_* \to \pi^\rig_* s_*(\MM_{P,S}),$$
likewise for $Q_*$ and $M_*$.
These are $\MU_*$-module maps ($s_* X$ has the
structure of an $s_*(\MGL)$-module, $X$ a $\MGL$-module).

For a $\MU_*$-module $N_*$ which is torsionfree
as abelian group we informally
denote by
$s_*(\MGL) \otimes_{\MU_*} N_*$ the module
in $\SH(S)^\integers$ which has
$\Sigma_T^q \HH \otimes N_q$ in the $q$-th component,
similarly for maps between such $\MU_*$-modules.
By Lemma (\ref{tors-ten}) the module $s_*(\MGL) \otimes_{\MU_*} N_*$ has the
weak universal property that for a given map
of $\MU_*$-modules $\phi \colon N_* \to \pi^\rig_* s_*(\NN')$, $\NN'$
a $\MGL$-module, there is an induced map
$$s_*(\MGL) \otimes_{\MU_*} N_* \to s_*(\NN'),$$
compatible with $\phi$ unique up to possible phantoms.

Thus
we get maps $$\psi_P \colon s_*(\MGL) \otimes_{\MU_*} P_* \to s_*(\MM_{P,S}),$$
similarly $\psi_Q$ and $\psi_M$ for $Q_*$ and $M_*$. The
maps $\psi_P$ and $\psi_Q$ are isomorphisms by the
assumption (SlMGL) and since $P_*$ and $Q_*$ are retracts
of free $\MU_*$-modules. Via these isomorphisms the map
$$s_*(\MM_{P,S}) \to s_*(\MM_{Q,S})$$ represents the map
$$s_*(\MGL) \otimes_{\MU_*} (P_* \to Q_*).$$ Now since $M_*$
is torsionfree the cofiber of the last map
is $s_*(\MGL) \otimes_{\MU_*} M_*$ (this is already so
for the cofibers of the maps $\S^\Top \otimes (P_q \to Q_q)$ in
the topological stable homotopy category).

This shows that the map $\psi_M \colon s_*(\MGL) \otimes_{\MU_*} M_*
\to s_*(\CC_S)$ is an isomorphism.
This is the content of the following proposition.

\begin{prop} \label{spec-case}
Theorem (\ref{sl-thm}) holds for Landweber spectra
of the form $\MGL \wedge \EE$ for $\EE$ Landweber.
\end{prop}

\begin{rem}
Consider the base change of the boundary map $\CC_\integers \to M_P[1]$ of
the triangle (\ref{aufl-tri}) to the spectrum $S$ of a subfield
of $\complex$. It is phantom in $\SH(S)_\caT$ since the corresponding homology
theories yield a short exact sequence. In general it is non-trivial
since after topological realization we recover the original
sequence $P_* \to Q_* \to M_*$ as coefficients,
and $M_*$ is in general not projective.
\end{rem}

\section{Cosimplicial resolutions}
\label{cosimp-reso}

In this section we prove theorem (\ref{sl-thm}).

Let $\bigtriangleup$ be the simplicial category,
$\bigtriangleup_*$ the category of the ordered {\em pointed}
sets $[n]_* = \{0,\ldots,n \} \coprod \{*\}$ for $n \in \{-1,0,1, \ldots \}$
pointed by $*$ and order preserving pointed maps.
An extension of a cosimplicial diagram to $\bigtriangleup_*$
corresponds to a `contraction' to the value at $[-1]_*$.
For example the homotopy limit of a cosimplicial diagram
which is the restriction of a $\bigtriangleup_*$-diagram in
a model category is weakly equivalent to the value at $[-1]_*$.
We shall only need the following strict version of the assertion.

\begin{lem} \label{cosimp-res}
Let $\psi_\bullet \colon A_\bullet \to B_\bullet$ be a map between $\bigtriangleup_*$-diagrams
in a category. Suppose $\psi_\bullet$ is an isomorphism
on the objects $[i]_*$ of $\bigtriangleup_*$ for $i \ge 0$.
Then $\psi_{-1}$ is also an isomorphism.
\end{lem}
\begin{proof}
We let $g \colon A_{-1} \to A_0$, $h \colon A_0 \to A_{-1}$ be the maps
induced by the unique maps in $\bigtriangleup_*$, $f,e \colon A_0 \to A_1$ the
maps induced by the maps $[0]_* \to [1]_*$ which send $0$ to $0$ resp. $1$, $k
\colon A_1 \to A_0$ the map induced by the map $[1]_* \to [0]_*$ sending $0$
to $0$ and $1$ to $*$. It is easily seen that these maps furnish a split equalizer.
Hence $A_{-1}$ is the limit of $A_\bullet |_\bigtriangleup$, likewise for
$B_\bullet$. The result follows.
\end{proof}

Let us fix a Landweber exact $\MGL$-module $\EE$
giving rise to a Landweber homology theory for the $\MU_*$-module
$M_*$. Let $\EE^\Top$ be the topological Landweber spectrum.

The cosimplicial resolution $\MGL^{\wedge \bullet} \wedge \EE$
of $\EE$ extends to a functor $\bigtriangleup_* \to \caD_{\MGL}$
using the $\MGL$-module structure on $\EE$.
The wedge $\MGL^{\wedge i} \wedge \EE$
is regarded as $\MGL$-module via the last factor.

We have natural maps
$$\pi_{2j} (\MU^{\wedge i} \wedge \EE^\Top) \to
\pi^\rig_j s_*(\MGL^{\wedge i} \wedge \EE)$$
which induce maps
$$\Sigma_T^j \HH \otimes \pi_{2j} (\MU^{\wedge i} \wedge \EE^\Top) \to
s_j(\MGL^{\wedge i} \wedge \EE)$$
which are unique up to possible phantoms.

These maps are also functorial in $i$ up to possible
phantom maps. More precisely we have a $\bigtriangleup_*$-diagram
$\Sigma_T^j \HH \otimes \pi_{2j} (\MU^{\wedge \bullet} \wedge \EE^\Top)$
in $\SH(S)$ modulo phantoms
and a transformation of $\bigtriangleup_*$-diagrams
$$\Sigma_T^j \HH \otimes \pi_{2j} (\MU^{\wedge \bullet} \wedge \EE^\Top) \to
s_j(\MGL^{\wedge \bullet} \wedge \EE),$$
again well defined up to possible phantoms.

This induces a transformation of diagrams of cohomology theories
defined on compact objects of $\SH(S)$
$$\Hom(-,\Sigma_T^j \HH \otimes \pi_{2j} (\MU^{\wedge \bullet}
\wedge \EE^\Top))=$$
$$\Hom(-,\Sigma_T^j \HH) \otimes \pi_{2j} (\MU^{\wedge \bullet}
\wedge \EE^\Top)
\to \Hom(-,s_j(\MGL^{\wedge \bullet} \wedge \EE)).$$
By Proposition \ref{spec-case} we
know that this is an isomorphism
on the subcategory of $\bigtriangleup_*$ spanned by the
objects $\{[0]_*, [1]_*, \ldots \}$.
By Lemma (\ref{cosimp-res})
it follows that it is also an isomorphism
on $[-1]_*$, which is Theorem (\ref{sl-thm}).

\begin{rem}
One can try to streamline the argument in the second step
by showing that $\HH$ can be realized as an $E_\infty$-algebra.
First note that $s_0$ can be obtained by colocalization along
all $\{\Sigma^{p,q} \Sigma^\infty_+ X | q \ge 0\}$
and then localization along the maps $\caS=\{\Sigma^{p,q} \Sigma^\infty_+ X
\to 0 | q > 0\}$. There is the problem that the colocalization
might not be cofibrantly generated, hence we cannot apply the
techniques available to persue the further localization.
Instead one looks at the full $\infty$-subcategory of the
$\infty$-category associated to the semimodel category of
$E_\infty$-ring spectra whose underlying objects are effective.
This is presentable in the sense of \cite{lurie-topoi}
an thus one should be able to find
a left proper combinatorial model. Then one can
directly localize this model category of effective $E_\infty$-ring
spectra along the free $E_\infty$-maps generated by $\caS$.
Alternatively one can try to localize the $\infty$-category directly.
A local model with respect to this localization yields
$\HH$ as an $E_\infty$-algebra under $\MGL$.

Having this one can form the derived category of $\HH$-modules
$\caD_\HH$ and using in the arguments of this paragraph
that a map between strictly $\HH$-cellular objects
in $\caD_\HH$ (with the definition of being strictly $\HH$-cellular
altered to be generated by $\HH$ inside $\caD_\HH$)
is an isomorphism if it is so on the $\pi_{i,0}$, $i \in \integers$.
\end{rem}

\section{Cohomological Landweber Exactness}

We start again with a topological evenly graded Landweber spectrum $\EE^\Top$
and let $M_*=\EE^\Top$ be the coefficients.
Let $\EE \in \caD_\MGL$ be the corresponding Landweber module.
It is well defined up to phantoms in $\caD_{\MGL,\caT}$.
We also denote by
$\EE$ the underlying spectrum in $\SH(S)_\caT$ with the $\MGL$-module structure
in $\SH(S)$.

\begin{lem}
 The functor $v \colon \caD_{\MGL,\caT} \to \SH(S)_\caT$ preserves phantom maps.
\end{lem}
\begin{proof}
 For $X \in \SH(S)_{\caT,f}$ and $E \in \caD_{\MGL,\caT}$ we have $\Hom(X,v E) = \Hom(\MGL \wedge X,E)$.
\end{proof}

We want to exhibit a natural map

$$\alpha_{M_*,X} \colon \MGL^{**} X \otimes_{\MU^*} M^* \to \EE^{**} X$$

for any $X \in \SH(S)$. As usual $M^*=M_{-*}$.

Therefore let $a \in \MGL^{p,q} X$ and $b \in M^i$.
By smashing the map $a \colon \Sigma^{-p,-q} X \to \MGL$ with $\EE$ and applying
the module structure map we get a map $\Sigma^{-p,-q} X \wedge \EE \to \EE$.
Composing with $b \colon \one^{-2i,-i} \to \EE$ we get a map
$\Sigma^{-2i-p,-i-q} X \to \EE$. This defines the map $\alpha_{M_*,X}$.

Let $N_*$ be other Landweber coefficients and $M_* \to N_*$ a $\MU_*$-map.
Let $\FF$ be the motivic spectrum corresponding to $N_*$ derived from a $\MGL$-module
and $f \colon \EE \to \FF$ be a map of $\MGL$-modules in $\SH(S)$ corresponding to $M_* \to N_*$. It is unique
up to possible phantoms in $\SH(S)_\caT$.

From the definition of $\alpha_{M_*,X}$ and $\alpha_{N_*,X}$ it follows
that these maps are natural in $M_* \to N_*$ and $f$.

It follows that we get a transformation of $\bigtriangleup_*$-diagrams
$$\alpha_{(\EE^\Top \wedge \MU^{\wedge \bullet})_*,X} \colon \MGL^{**} X \otimes_{\MU^*} (\EE^\Top \wedge \MU^{\wedge \bullet})^*
\to (\EE \wedge \MGL^{\wedge \bullet})^{**} X.$$

\begin{lem}
 $\alpha_{(\EE^\Top \wedge \MU^{\wedge \bullet})_*,X}$ is an isomorphism for compact $X$ and $\bullet >0$.
\end{lem}
\begin{proof}
 Clearly it is sufficient to prove the statement for $\bullet=1$. Let $N_*=(\EE^\Top \wedge \MU)_*$ be the
coefficients of $\EE \wedge \MGL$. Here we view $N_*$ as $\MU_*$-module via the last factor.
As already remarked $N_*$ is flat as $\MU_*$-module. This can be seen e.g. by considering $M_*$
as flat quasi coherent sheaf on the moduli stack of formal groups with trivialized constant vector fields.
Then $N_*$ is just the pullback of this sheaf to $\Spec(\MU_*)$.

Let
$$0 \to P_* \overset{\phi}{\to} Q_* \to N_* \to 0,$$
be a resolution by projective $\MU_*$-modules as in section (\ref{first-sl}).

Then
$$0 \to \MGL^{**} X \otimes_{\MU^*} P_* \to \MGL^{**} X \otimes_{\MU^*} Q_* \to
\MGL^{**} X \otimes_{\MU^*} N_*$$ is again exact by the flatness of $N_*$.
Moreover $\alpha_{P_*,\_}$, $\alpha_{Q_*,\_}$ are easily seen to be isomorphisms on compacts.
Thus the map induced by $\phi$ on the targets of these maps is injective on compacts.
Since this is part of the long exact cohomology sequence for the triangle corresponding to the resolution
we deduce that the target of $\alpha_{N_*,\_}$ is the cokernel of the above injection on compacts.
This proves the claim.
\end{proof}

\begin{cor}
 $\alpha_{M_*,\_}$ is an isomorphism between cohomology theories defined on compact objects.
\end{cor}

We also deduce the following uniquness statement:

\begin{cor}
The phantom maps in $\SH(S)_\caT$ coming from $\caD_{\MGL_\integers,\caT}$ up to which
the Landweber spectrum $\EE$ is well-defined are also phantom in $\SH(S)$.
\end{cor}

\bibliographystyle{plain}

\bibliography{slice}

\begin{center}
Fakult{\"a}t f{\"u}r Mathematik, Universit{\"a}t Regensburg, Germany.\\
e-mail: Markus.Spitzweck@mathematik.uni-regensburg.de
\end{center}

\end{document}